\documentclass{amsart}
\usepackage{amsfonts}
\usepackage{amsmath}
\usepackage{amssymb}
\usepackage{hyperref}
\usepackage{enumerate}

\numberwithin{equation}{section}

\newtheorem{theorem}{Theorem}[section]
\newtheorem{prop}[theorem]{Proposition}
\newtheorem{lemma}[theorem]{Lemma}
\newtheorem{rem}[theorem]{Remark}
\newtheorem{exmp}{Example}
\newtheorem{cor}[theorem]{Corollary}
\newcounter{casec}
\newenvironment{case}[1]{\addtocounter{casec}{1} {\bf Case} \arabic{casec}: #1.}{\newline}

\newcommand{\hol}[3][]{\chi_{#1} \left(#2, #3\right)}
\newcommand{\holun}[2]{\chi \left(#1, #2\right)}
\newcommand{\G}[2][\mathcal H]{\mathcal G_{#2} \left(#1\right)}
\newcommand{\PH}[2][\mathcal H]{\mathcal P_{#2} \left(#1\right)}
\newcommand{\bh}[1][\mathcal H]{\mathcal B \left(#1\right)}
\newcommand{\sh}[1][\mathcal H]{\mathcal S \left(#1\right)}
\newcommand{\fs}[1][\mathcal H]{\mathcal F_s \left(#1\right)}
\DeclareMathOperator{\tr}{tr}
\DeclareMathOperator{\ran}{ran}
\DeclareMathOperator{\rk}{rank}
\DeclareMathOperator{\re}{Re}
\DeclareMathOperator{\diag}{diag}
\newcommand{\dg}[2]{d (#1, #2)}
\newcommand{\NN}{\mathbb N}
\newcommand{\RR}{\mathbb R}
\newcommand{\CC}{\mathbb C}
\newcommand{\ha}{\mathbb H}
\newcommand{\HH}{\mathcal H}
\newcommand{\KK}{\mathcal K}
\newcommand{\mm}[1][m]{\mathcal M_{#1}}
\newcommand{\un}[1][m]{\mathcal U_{#1}}
\newcommand{\ih}[1][m]{\mathcal I \mathcal H_{#1}}
\newcommand{\hz}[1][m]{\mathcal H_{0, #1}}
\newcommand{\ihz}[1][m]{\mathcal I \mathcal H_{0, #1}}

\begin{document}

\title{Linear preservers of rank $k$ projections}
\author{
	Lucijan Plevnik}
\address{University of Ljubljana, Faculty for Mathematics and Physics, Jadranska ulica 19, Ljubljana, Slovenia}
\email{
	lucijan.plevnik@fmf.uni-lj.si}

\begin{abstract}
Let $\HH$ be a complex Hilbert space and $\fs$ the real vector space of all self-adjoint finite rank bounded operators on $\HH$. We generalize the famous Wigner's theorem by characterizing linear maps on $\fs$ which preserve the set of all rank $k$ projections. In order to do this, we first characterize linear maps on the real vector space $\hz[2k]$ of trace zero $(2k) \times (2k)$ hermitian matrices which preserve the subset of unitary matrices in $\hz[2k]$.

We also study linear maps from $\fs$ to $\fs[\KK]$ sending projections of rank $k$ to finite rank projections. We prove some properties of such maps, e.g. that they send rank $k$ projections to projections of a fixed rank. We give the complete description of such maps in the case $\dim \HH = 2$. We give several examples which show that in the general case the problem to describe all such maps seems to be complicated.
\end{abstract}

\maketitle

\section{Introduction}

Let $\HH$ be a complex Hilbert space. We say that a bounded linear operator on $\HH$ is a projection if it is idempotent and self-adjoint. In this paper, we will study linear maps which send projections of a fixed finite rank to finite rank projections. It is well-known that if $k \in \NN$, $k < \dim \HH$, then the real linear span of the set of all rank $k$ projections $\PH{k}$ equals the real linear space of finite rank bounded self-adjoint operators $\fs$, see e.g. \cite[Lemma 2.1.5.]{Molnar}. Thus, it makes sense to assume that the considered linear maps map from $\fs$ to $\fs[\KK]$.

Our motivation to study such maps has roots in the famous Wigner's theorem which plays a central role in mathematical foundations of quantum mechanics. It was published in 1931 and one can find the English translation in \cite{wig}. The theorem can be formulated as follows.
\begin{theorem}[Wigner]
	Suppose that $\dim \HH > 1$ and let $\phi \colon \PH{1} \to \PH{1}$ be a bijective map such that
	\begin{eqnarray}\label{trpq}
		\tr \phi(P) \phi(Q) = \tr PQ, & P, Q \in \PH{1}.
	\end{eqnarray}
	Then there exists a unitary or an anti-unitary operator $U \colon \HH \to \HH$ such that
	\begin{eqnarray*}
		\phi(P) = UPU^\ast, & P \in \PH{1}.
	\end{eqnarray*}
\end{theorem}

This theorem has been generalized in several ways. First, note that the assumption of injectivity is redundant since it clearly follows from \eqref{trpq}. It turns out that neither surjectivity is needed in order to get an analogous result, only the isometry $U$ needs not be surjective, see e.g. \cite{geh-wig}.

Another way to generalize Wigner's theorem is to describe preservers of a quantity on $\PH{k}$ which is equivalent to the transition probability $\tr PQ$ in the case $k=1$. Such an example is the description of surjective isometries on $\PH{k}$ with respect to the gap metric, i.e. the metric induced by the operator norm on $\PH{k}$, due to Geh\'er and Šemrl \cite{gehse}. In Section \ref{hola}, we will see that this metric is relevant in our study. Moreover, in the case $k=1$ it equals $\sin \varphi$, where $\varphi$ is the angle between the one-dimensional subspaces $\ran P$ and $\ran Q$, while $\tr PQ$ equals $\cos^2 \varphi$. Here, $\ran A$ denotes the range of an operator $A$.

Let us introduce the result which is even more important for our study. The concept of an angle between one-dimensional subspaces can be generalized by the concept of principal angles between finite-dimensional subspaces. This was introduced by Jordan \cite{jor} and is related to the Halmos' two projections theorem. We will describe this topic more in Section \ref{hola}. At this point, we mention Moln\'ar's characterization of preservers of principal angles.

\begin{theorem}[L. Moln\'ar, \cite{mol1}, \cite{mol2}]
	Suppose that $k \in \NN$, $k < \dim \HH$, and let $\phi \colon \PH{k} \to \PH{k}$ be a map such that for all $P, Q \in \PH{k}$, principal angles between $\phi(P)$ and $\phi(Q)$ are the same as those between $P$ and $Q$. Then there exists a linear or conjugate-linear isometry $U \colon \HH \to \HH$ such that either
	\begin{eqnarray*}
		\phi(P) = UPU^\ast, & P \in \PH{k},
	\end{eqnarray*}
	or $k \ge 2$, $\dim \HH = 2k$, and
	\begin{eqnarray*}
		\phi(P) = U(I-P)U^\ast, & P \in \PH{k}.
	\end{eqnarray*}
\end{theorem}

When proving this theorem, Moln\'ar proved that a map $\phi \colon \PH{k} \to \PH{k}$ preserving principal angles can be uniquely extended to an injective linear map $\phi \colon \fs \to \fs$. In fact, his lemma shows that in order to get the linear extension, it is enough for $\phi$ to preserve the quantity $\tr PQ$ which equals the sum of squares of cosines of the principal angles for $P, Q \in \PH{k}$. Later, Geh\'er used this lemma to show that preservation of $\tr PQ$ alone is enough to get the same result as when preserving all principal angles.

\begin{theorem}[G. P. Geh\'er, \cite{geh-grass}]
	Suppose that $k \in \NN$, $k < \dim \HH$, and let $\phi \colon \PH{k} \to \PH{k}$ be a map such that
	\begin{eqnarray*}
		\tr \phi(P) \phi(Q) = \tr PQ, & P, Q \in \PH{k}.
	\end{eqnarray*}
	Then there exists a linear or conjugate-linear isometry $U \colon \HH \to \HH$ such that either
	\begin{eqnarray*}
		\phi(P) = UPU^\ast, & P \in \PH{k},
	\end{eqnarray*}
	or $k \ge 2$, $\dim \HH = 2k$, and
	\begin{eqnarray*}
		\phi(P) = U(I-P)U^\ast, & P \in \PH{k}.
	\end{eqnarray*}
\end{theorem}

This motivated Pankov to study linear maps $\fs \to \fs[\KK]$ which send projections of a fixed finite rank to projections of a fixed finite rank. The standard examples of such maps are the following.

If $A$ is a complex matrix or a bounded linear or conjugate-linear bounded operator between complex Hilbert spaces, then $A^\ast$ denotes its conjugate transpose or adjoint, respectively.

\begin{exmp}\label{stu}
	Let $U \colon \HH \to \KK$ be a linear or conjugate-linear isometry. Then the map
	$$
	A \mapsto UAU^\ast
	$$
	is an injective linear map $\fs \to \fs[\KK]$ which maps $\PH{k}$ into $\PH[\KK]{k}$.
\end{exmp}

For each $m \in \NN$ denote by $\HH_m$ the real vector space of $m \times m$ complex hermitian matrices and by $I_m \in \HH_m$ the identity matrix. For every $k \in \{0, 1, \ldots, m\}$ denote by $\mathcal P_{m,k} \subset \HH_m$ the set of all rank $k$ projections.

\begin{exmp}\label{n-k}
	Let $k, m \in \NN$, $k < m$. Define $L_k \colon \HH_m \to \HH_m$ by
	\begin{eqnarray*}
		L_k (A) = \tfrac{\tr A}{k} I_m - A, & A \in \HH_m.
	\end{eqnarray*}
	Then $L_k$ is a bijective linear map such that $L_k^{-1} = L_{m-k}$ and $L_k \left(\mathcal P_{m,k}\right) = \mathcal P_{m,m-k}$.
\end{exmp}

The next example shows that in the case $m = 2$, the map from Example \ref{n-k} is of the standard form from Example \ref{stu}.

For each $m \in \NN$ denote by $\un$ the set of all $m \times m$ complex unitary matrices. If $A$ is a complex matrix, then $\overline{A}$ denotes its componentwise complex conjugate.

\begin{exmp}\label{lk2k}
	For $U = \begin{bmatrix}
		0 & 1 \\
		-1 & 0
	\end{bmatrix} \in \un[2]$ we have
	\begin{eqnarray*}
		(\tr A) I_2 - A = U \overline{A} U^\ast, & A \in \HH_2.
	\end{eqnarray*}
\end{exmp}

The fruition of Pankov's study is the following result.

\begin{theorem}[M. Pankov, \cite{pankov_km}]\label{ppt}
	Let $\phi \colon \fs \to \fs$ be a linear map such that $\phi \left(\PH{k}\right) \subset \PH{k}$ and $\phi|_{\PH{k}}$ is injective. Then there exists a linear or conjugate-linear isometry $U \colon \HH \to \HH$ such that either
	\begin{eqnarray*}
		\phi(A) = UAU^\ast, & A \in \fs,
	\end{eqnarray*}
	or $k \ge 2$, $\dim \HH = 2k$, and
	\begin{eqnarray*}
		\phi(A) = U\left(\tfrac{\tr A}{k}\mathcal I_\HH-A\right)U^\ast, & A \in \fs.
	\end{eqnarray*}
\end{theorem}
 Here, $\mathcal I_{\HH}$ denotes the identity operator on $\HH$. He also considered linear maps $\fs \to \fs[\KK]$ which are injective on the set $\PH k$ and send projections of rank $k$ to projections of rank $m$. He showed that if $\HH$ is infinite-dimensional, then under the additional assumption
\begin{eqnarray*}
	\dim \left(\ran \phi(P) \cap \ran \phi(Q)\right) \ge m-k, & P, Q \in \PH{k},
\end{eqnarray*}
such maps are direct sums of the standard maps in Examples \ref{stu} and \ref{n-k}, and a non-injective map from the following example.

Denote by $\PH{f} \subset \fs$ the set of all finite rank projections. For each $A \in \fs$ denote its trace by $\tr A$.

\begin{exmp}\label{con}
	Let $P_0 \in \PH[\KK]{f}$. Then the map
	$$
	A \mapsto \tfrac{\tr A}{k} P_0
	$$
	is a linear map $\fs \to \fs[\KK]$ which maps $\PH{k}$ constantly into $\PH[\KK]{f}$.
\end{exmp}

However, the general problem to characterize all linear maps $\fs \to \fs[\KK]$ which send projections of a fixed rank to projections of a fixed rank without additional assumptions is much more difficult. Indeed, let us present some examples of such maps which are not of the above described standard forms in the case when $\HH$ is finite-dimensional. Some of these examples also fail to have the property of the so far presented examples that the map is either injective or constant on the set $\PH{k}$.

We will denote the set of all $m \times n$ complex matrices by $\mm[m,n]$ and write shortly $\mm = \mm[m,m]$.

\begin{exmp}\label{rlunun}
	For all $k \in \NN$ define $\rho_k \colon \CC^k \to \mm[2^{k-1}]$ in a recursive way:
	\begin{eqnarray*}
		\rho_1(z) = z, & z \in \CC,
	\end{eqnarray*}
	and for $k \ge 2$:
	\begin{eqnarray*}
		\rho_k \left(\left[\begin{array}{c}
			z \\ w
		\end{array}\right]\right)
		= \begin{bmatrix}
			z I_{2^{k-2}} & \rho_{k-1} (w) \\
			- \rho_{k-1} (w)^\ast& \overline{z} I_{2^{k-2}}
		\end{bmatrix}, & z \in \CC, \, w \in \CC^{k-1}.
	\end{eqnarray*}
	Then $\rho_k$ is a real linear map and $\rho_k \left(\{v \in \CC^k \mid \|v\|=1\}\right) \subset \un[2^{k-1}]$.
\end{exmp}

\begin{exmp}
	Let $k, n \in \NN$, $k < n$. Define
	\begin{eqnarray*}
		\phi\left(\begin{bmatrix}
			A & v \\
			v^\ast & b
		\end{bmatrix}\right) = 
		\begin{bmatrix}
			\left(\tfrac{\tr A +b}{k} - b\right) I_{2^{n-2}} & \rho_{n-1}(v) \\
			\rho_{n-1}(v)^\ast & b I_{2^{n-2}}
		\end{bmatrix}, & A \in \HH_{n-1}, \, v \in \CC^{n-1}, \, b \in \RR,
	\end{eqnarray*}
	where $\rho_{n-1}$ is as in Example \ref{rlunun}. Then $\phi \colon \HH_n \to \HH_{2^{n-1}}$ is a linear map such that $\phi \left(\mathcal P_{n,k}\right) \subset \mathcal P_{2^{n-1},2^{n-2}}$ and $\phi|_{\mathcal P_{n,k}}$ is not constant. If $n \ge 3$, then $\phi|_{\mathcal P_{n,k}}$ is not injective.
\end{exmp}

\begin{exmp}\label{ev}
	Let $m \in \NN$, $v_0 \in \CC^m$, $\|v_0\| = 1$, and $\mathcal M \subset \mm$ a real vector subspace. Let $\rho_m \colon \CC^m \to \mm[2^{m-1}]$ be as in Example \ref{rlunun}. Then the map $\phi \colon \mathcal M \to \mm[2^{m-1}]$, given by
	\begin{eqnarray*}
		\phi(A) = \rho_m \left(Av_0\right), & A \in \mathcal M,
	\end{eqnarray*}
	is a real linear map such that $\phi(\mathcal M \cap \un) \subset \un[2^{m-1}]$.
\end{exmp}

For all $m \in \NN$ denote
$$
\hz = \{A \in \HH_m \mid \tr A = 0\}
$$
and $\ihz = \hz \cap \un$.

\begin{exmp}\label{e2km}
	Let $k, m \in \NN$ and $t \in [0,1]$. If there exists a real linear map $\tau \colon \hz[2k] \to \mm$ such that
	\begin{equation}\label{taiu}
		\tau \left(\ihz[2k]\right) \subset \un,
	\end{equation}
	then the map $\phi \colon \HH_{2k} \to \HH_{2m}$ defined by
	\begin{eqnarray*}
		\phi(A) = \begin{bmatrix}
			t \tfrac{\tr A}{k} I_m  & \sqrt{t(1-t)} \tau \left(2A - \tfrac{\tr A}{k} I_{2k}\right)  \\
			\sqrt{t(1-t)} \tau \left(2A - \tfrac{\tr A}{k} I_{2k}\right)^\ast  & (1-t) \tfrac{\tr A}{k} I_m
		\end{bmatrix}, & A \in \HH_{2k},
	\end{eqnarray*}
	is linear and we have $\phi \left(\mathcal P_{2k,k}\right) \subset \mathcal P_{2m,m}$. Such defined $\phi$ is injective if and only if $\tau$ is injective. By Example \ref{ev}, $\tau$ satisfying \eqref{taiu} exists when $m = 2^{2k-1}$ and it is not injective if $k \ge 2$. We will see later that:
	\begin{itemize}
		\item If $k = 1$, then $\tau$ satisfying \eqref{taiu} exists if and only if $m$ is even. In this case, such a $\tau$ is injective and of a standard form (Proposition \ref{h0u}).
		\item For any $\varphi \in \RR$, the linear map given by
		\begin{eqnarray*}
			\tau(A) = \cos \varphi \, A \otimes I_{2k} + i \sin \varphi \, I_{2k} \otimes A, & A \in \hz[2k],
		\end{eqnarray*}
		is an example of injective $\tau \colon \hz[2k] \to \un$ satisfying \eqref{taiu} for $m = 4k^2$. If $k \ge 2$ and $\varphi$ is not an integer multiple of $\tfrac{\pi}{2}$, then $\tau$ does not preserve singular values, so it is not of the standard form from Proposition \ref{h0u}.
		\item $\tau$ satisfying \eqref{taiu} may exist only when $m \ge 2k$ (Proposition \ref{hura}).
	\end{itemize}
\end{exmp}

Next, we will show in Proposition \ref{fkm} that if $\phi \colon \fs \to \fs[\KK]$ is a linear map, then it is enough to assume $\phi \left(\PH{k}\right) \subset \PH[\KK]{f}$ since this assumption implies that $\phi \left(\PH{k}\right) \subset \PH[\KK]{m}$ for some $m$. Finally, we will give a quite simple complete description of such maps in the case $\dim \HH = 2$. In order to explain this result, let us introduce a couple of more examples.

Denote by $\HH \otimes \KK$ the tensor product of Hilbert spaces. If $A$ and $B$ are bounded linear operators on $\HH$ and $\KK$, respectively, then $A \otimes B$ denotes their tensor product. In the finite-dimensional case, $A \otimes B$ will denote the Kronecker product of matrices $A$ and $B$.

\begin{exmp}\label{tenp}
	Let $P_0 \in \PH[\KK]{f}$. The map
	$$
	A \mapsto A \otimes P_0
	$$
	is a linear map $\fs \to \fs[\HH \otimes \KK]$ which maps $\PH{k}$ into $\PH[\HH \otimes \KK]{kr}$, where $r = \rk P_0$. If $P_0 \ne 0$, then this map is injective.
\end{exmp}

Denote by $\sh$ the real vector space of all bounded self-adjoint linear operators on $\HH$ and by $\PH{} \subset \sh$ the set of all projections.

\begin{exmp}\label{pq}
	Let $P_0, Q_0 \in \PH[\KK]{}$ and define $\psi \colon \fs \to \sh[\HH \otimes \KK]$ by
	\begin{eqnarray*}
		\psi(A) = A \otimes P_0 + \left(\tfrac{\tr A}{k} \mathcal I_\HH - A\right) \otimes Q_0, & A \in \fs.
	\end{eqnarray*}
	Then:
	\begin{itemize}
		\item $\psi$ is a linear map such that $\psi \left(\PH{k}\right) \subset \PH[\HH \otimes \KK]{}$.
		\item If $\dim \HH = n \in \NN$ and $P_0, Q_0 \in \PH[\KK]{f}$, then $\psi$ maps into $\fs[\HH \otimes \KK]$ and $\psi \left(\PH{k}\right) \subset \PH[\HH \otimes \KK]{m}$, where $m = k \rk P_0 + (n-k) \rk Q_0$.
		\item If $P_0 = Q_0$, then $\psi$ is as in Example \ref{con}.
		\item If $P_0 \ne Q_0$, then $\psi$ is injective.
		\item If $Q_0 = 0$, then $\psi$ is as in Example \ref{tenp}.
		\item If $P_0 = 0$, then $\psi$ is a composition of maps in Example \ref{tenp} and Example \ref{n-k}.
	\end{itemize}
\end{exmp}

We will show in Theorem \ref{2f} that in the case $\dim \HH = 2$, all linear maps $\fs \to \fs[\KK]$ which send projections of a fixed rank to projections of a finite rank are combinations of maps in Examples \ref{stu}, \ref{con}, and \ref{pq}.

In \cite{pp}, Pankov and Plevnik considered what happens if one omits the injectivity assumption in Theorem \ref{ppt}. They showed that in the very special case $k=1$, the only additional resulting map is the one from Example \ref{con}. The main result of this paper is that the same conclusion holds in the general case, see Theorem \ref{phkk}. In order to prove this result, we will first prove Theorem \ref{iho} which is interesting in itself and characterizes linear maps $\hz[2k] \to \hz[2k]$ which preserve the set $\ihz[2k]$. We will show that all such maps are unitary similarities, possibly composed by the entry-wise complex conjugation or multiplication by $-1$.

\section{The set $\hol[a]{X}{Y}$}\label{hola}

Every element $P \in \PH{}$ can be identified with its range. We denote by $\G{}$ the set of all closed subspaces of $\HH$, its subset of finite-dimensional subspaces by
$$
\G f = \{ X \in \G{} \mid \dim X < \infty \},
$$
and the Grassmann space of subspaces of a fixed finite dimension by
$$
\G k = \{ X \in \G{} \mid \dim X = k \}.
$$
For each $X \in \G{}$, $P_X \in \PH{}$ and $X^\bot$ denote the projection with range $X$ and the orthogonal complement of $X$, respectively.

For any $X, Y \in \G{f}$ and $a \in \RR$ define the set
$$
\hol[a]{X}{Y} := \left\{ Z \in \G{f} \, : \, a (P_X + P_Y) + \left(1-2a\right) P_Z \in \PH{f} \right\}.
$$

In \cite[Lemma 2]{geh-grass}, it was characterized, when $\hol[1]{X}{Y}$ is a one-dimensional manifold for $X, Y \in \G k$. The aim of this section is to precisely describe the set $\hol[a]{X}{Y}$ for $X, Y \in \G f$. In order to understand this set better, we will use the Wedin's version of the Halmos' two projections theorem. One can obtain this formulation by applying \cite[Corollary 2.2]{bsp}.

If $\HH$ is of dimension $m \in \NN$ and $\mathcal B$ is its orthonormal basis, then for any $A \in \sh$ denote the matrix representation of $A$ according to $\mathcal B$ by $A_{\mathcal B} \in \HH_m$.

\begin{prop}\label{halmd}
	Let $X, Y \in \G{f}$. Then there exist an orthonormal basis $\mathcal B$ of $X+Y$ such that
	$$
	(P_X|_{X+Y})_{\mathcal B} = I_m \oplus I_p \oplus 0_q \oplus \bigoplus_{j=1}^d \left[\begin{array}{cc}
		1 & 0 \\ 0 & 0
	\end{array}\right]
	$$
	and
	$$
	(P_Y|_{X+Y})_{\mathcal B} = I_m \oplus 0_p \oplus I_q \oplus \bigoplus_{j=1}^d \left[\begin{array}{cc}
		\cos^2 \varphi_j & \cos \varphi_j \sin \varphi_j \\ \cos \varphi_j \sin \varphi_j & \sin^2 \varphi_j
	\end{array}\right]
	$$
	for some $m, p, q, d \in \NN \cup \{0\}$ and $\varphi_1, \ldots, \varphi_d \in \left(0,\tfrac{\pi}{2}\right)$.
\end{prop}

Note that $X$ and $Y$ are subspaces of the finite-dimensional complex vector space $X + Y$. In this context, $\varphi_1, \ldots, \varphi_d$ from Proposition \ref{halmd} are precisely principal angles between $X$ and $Y$ from the interval $\left(0,\tfrac{\pi}{2}\right)$, while $\tfrac{\pi}{2}$ and $0$ are principal angles of respective multiplicities $p+q$ and $m$. An interested reader can find more about principal angles e.g. in \cite{Galantai}, \cite{bsp} or \cite[Section VII.1]{bhat}.

We continue with some technical lemmas.

We will need the following property of solutions of the Sylvester equation. Denote the set of bounded linear operators on $\HH$ by $\bh$. For any $A \in \bh$, we denote its spectrum by $\sigma(A)$. Let $X \oplus Y$ be the direct sum of $X, Y \in \G{}$. Depending on the context, $A \oplus B$ may denote the direct sum of matrices or operators $A$ and $B$.

\begin{lemma}\label{abinv}
	Let $A, B \in \sh$ and $C \in \bh$. Let $M \in \G{}$ be invariant for $A, B, C$ such that $\sigma(A|_{M^\bot}) \cap \sigma(B|_{M}) = \emptyset$ and $B|_{M} \in \sh[M]$ is a compact operator. If $D \in \bh$ is such that $AD-DB = C$, then $M$ is invariant for $D$.
\end{lemma}

\begin{proof}
	According to the decomposition $\HH = M \oplus M^\bot$ we have
	\begin{eqnarray*}
		A = A_1 \oplus A_2, & B = B_1 \oplus B_2, & D = \begin{bmatrix}
			D_{11} & D_{12} \\
			D_{21} & D_{22}
		\end{bmatrix}.
	\end{eqnarray*}
	Since $M$ is invariant for $C$, we have $A_2D_{21} = D_{21}B_1$.
	Let $x \in M$ be an eigenvector for $B_1$ and let $\lambda$ be the corresponding eigenvalue. Then
	$$
	A_2 D_{21}x = \lambda D_{21}x.
	$$
	Since $\lambda \notin \sigma \left(A_2\right)$, we have $D_{21}x = 0$. Because $B_1$ is a compact self-adjoint operator  on $M$, we have $D_{21}x = 0$ for all $x \in M$.
\end{proof}

The following proposition is the main result of this section. We will modify the orthonormal basis from \ref{halmd} in such a way that the matrix of $P_X+P_Y$ will be diagonal.

Let us first introduce some notation. For $X, Y \in \G f$, the so-called gap metric is defined by
\begin{eqnarray*}
\dg{X}{Y} = \left\| P_X - P_Y \right\|,
\end{eqnarray*}
where $\|\cdot\|$ denotes the operator norm. By applying Proposition \ref{halmd}, we deduce that $d \left(X, Y\right)$ equals the sine of the largest principal angle between $X$ and $Y$.

For all $W, Z \in \G{}$ denote
$$
\left[W, Z\right] = \left\{ X \in \G{} \mid W \subset X \subset Z \right\}.
$$
For each $m \in \NN$, we will denote the set of all projections in $\HH_m$ by $\mathcal P_m$. We denote
$$
A \circ B = AB+BA
$$
if $A$ and $B$ are matrices or operators.

\begin{prop}\label{prophol}
	Let $X, Y \in \G f$ and $a \in \left(\tfrac{1}{2}, \infty\right)$. Then we have $\hol[a]{X}{Y} \ne \emptyset$ if and only if
	\begin{equation}\label{da}
		\dg{X}{Y} \le \tfrac{\sqrt{2a-1}}{a}.
	\end{equation}
	There exists an orthonormal basis $\mathcal B$ of $X+Y$ such that
	$$
	(P_X|_{X+Y})_{\mathcal B} = I_m \oplus I_p \oplus 0_q \oplus \bigoplus_{j=1}^r \begin{bmatrix}
		\tfrac{1+\cos \varphi_j}{2} I_{m_j}  & \tfrac{\sin \varphi_j}{2} I_{m_j} \\
		\tfrac{\sin \varphi_j}{2} I_{m_j} & \tfrac{1-\cos \varphi_j}{2} I_{m_j}
	\end{bmatrix}
	$$
	and
	$$
	(P_Y|_{X+Y})_{\mathcal B} = I_m \oplus 0_p \oplus I_q \oplus \bigoplus_{j=1}^r \begin{bmatrix}
		\tfrac{1+\cos \varphi_j}{2} I_{m_j}  & -\tfrac{\sin \varphi_j}{2} I_{m_j} \\
		-\tfrac{\sin \varphi_j}{2} I_{m_j} & \tfrac{1-\cos \varphi_j}{2} I_{m_j}
	\end{bmatrix}
	$$
	for some $m, p, q, r \in \NN \cup \{0\}$, $m_1, \ldots, m_r \in \NN$, and $0 < \varphi_r < \ldots < \varphi_1 < \tfrac{\pi}{2}$.

If \eqref{da} holds, then $\hol[a]{X}{Y}$ consists of those $Z \in [X \cap Y,X+Y]$, for which
\begin{equation}\label{pzmat}
	(P_Z|_{X+Y})_{\mathcal B} = I_m \oplus Q \oplus \bigoplus_{j=1}^r \begin{bmatrix}
	t_j I_{m_j}  & \sqrt{t_j\left(1-t_j\right)} U_j^\ast \\
	\sqrt{t_j\left(1-t_j\right)} U_j & \left(1-t_j\right) I_{m_j}
\end{bmatrix}
\end{equation}
for some $Q \in \mathcal P_{p+q}$ and $U_j \in \un[m_j]$, where
\begin{eqnarray*}
	t_j = \tfrac{(1+\cos \varphi_j) \left(a(1+\cos \varphi_j)-1\right)}{2(2a-1) \cos \varphi_j}, & j = 1, \ldots, r.
\end{eqnarray*}
\end{prop}

\begin{rem}
	If $a = 1$, then $t_j = \tfrac{1+\cos \varphi_j}{2}$ and \eqref{da} holds for all $X, Y \in \G{f}$.
\end{rem}

\begin{proof}
	We take the orthonormal basis of $X+Y$ from Proposition \ref{halmd} and first permute it and then transform it via the transition matrix
	$$
	I_{m+p+q} \oplus \bigoplus_{j=1}^r\left[ \begin{array}{cc}
		\cos \tfrac{\varphi_j}{2} I_{m_j} & \sin \tfrac{\varphi_j}{2} I_{m_j} \\
		\sin \tfrac{\varphi_j}{2} I_{m_j} & - \cos \tfrac{\varphi_j}{2} I_{m_j}
	\end{array}
	\right]
	$$
	to get the desired matrix representations of $P_X$ and $P_Y$.
	
	Note that \eqref{da} is equivalent to
	\begin{eqnarray}\label{pqt}
		(p+q>0 \implies a=1) & {\rm and} & 0 \le t_j \le 1, \, j = 1, \ldots, r.
	\end{eqnarray}
	
	If \eqref{pqt} holds and $Z \in \left[X \cap Y,X+Y\right]$ is as in \eqref{pzmat}, then it is straightforward to check that $Z \in \hol[a]{X}{Y}$.
	
	Suppose now that $\hol[a]{X}{Y}$ is not empty and let $Z$ be its element. Then $a(P_X+P_Y)+(1-2a)P_Z$ is a projection.	Denote $A = P_X+P_Y- \mathcal I_\HH$. Then $X+Y$ is invariant for $A$ and we have
	$$
	(A|_{X+Y})_{\mathcal B} = I_m \oplus 0_{p+q} \oplus \bigoplus_{j=1}^r \cos \varphi_j \left(I_{m_j} \oplus \left(-I_{m_j}\right)\right)
	$$
	and
	\begin{equation}\label{pxya}
	P_Z \circ A = A + \tfrac{1}{2a-1}(aA^2+(a-1)\mathcal I_\HH).
	\end{equation}
	Denote $M = \left(X \cap Y\right) \oplus \left(X+Y\right)^\bot$. Since $M^\bot \in \G{f}$ is invariant for $A$, \eqref{pxya} enables us to apply Lemma \ref{abinv} to deduce that $M^\bot$ is invariant for $P_Z \in \fs$, hence so is $M$. Since $A$ acts like $\mathcal I \oplus \left(-\mathcal I\right)$ on $M$, we obtain via a straightforward calculation from \eqref{pxya} that $P_Z$ acts like $\mathcal I \oplus 0$ on $M$. Thus, $Z \in [X \cap Y,X+Y]$. We apply \eqref{pxya} and Lemma \ref{abinv} again to deduce that 
	$$
	(P_Z|_{X+Y})_{\mathcal B} = I_m \oplus Q \oplus \bigoplus_{j=1}^r R_j
	$$
	for some $Q \in \mathcal P_{p+q}$, and $R_j \in \mathcal P_{2m_j}$, $j = 1, \ldots, r$, such that $0_{p+q} = (a-1) I_{p+q}$ and
	\begin{eqnarray*}
		\left(I_{m_j} \oplus \left(-I_{m_j}\right)\right) \circ R_j = 2 \left(\left(t_jI_{m_j}\right) \oplus \left(\left(t_j-1\right)I_{m_j}\right)\right), & j = 1, \ldots, r.
	\end{eqnarray*}
	Thus, \eqref{pqt} holds, and we get the form \eqref{pzmat}.
\end{proof}

\begin{cor}\label{dimhol}
	Let $a \in \left(\tfrac{1}{2}, \infty\right) \setminus \{1\}$ and $X, Y \in \G f$ such that $\hol[a]{X}{Y} \ne \emptyset$. Then there exists $k \in \NN \cup \{0\}$, $k \le \dim \HH$, such that $\{X, Y\} \cup \hol[a]{X}{Y} \subset \G k$.
\end{cor}

\begin{proof}
	Proposition \ref{prophol} implies that $\dg{X}{Y} \le \tfrac{\sqrt{2a-1}}{a} < 1$. Hence, $X \cap Y^\bot = Y \cap X^\bot = \{0\}$. Proposition \ref{prophol} now yields that $X$, $Y$, and all elements of $\hol[a]{X}{Y}$ have the same dimension.
\end{proof}

\section{Real linear maps transforming trace zero hermitian unitary matrices}

In this section, we will study real linear maps which transform trace zero hermitian unitary matrices to unitary matrices. These results will be relevant in the proof of our main results because of the connection which will be later revealed in Lemma \ref{a0p} and Lemma \ref{phraz}.

The crucial results in this section are Proposition \ref{hura} and Theorem \ref{iho}. In order to prove them, we need some technical lemmas.

For each $A \in \mm$ denote
\begin{eqnarray*}
	\re A = \tfrac{1}{2} \left(A+A^\ast\right) & {\rm and} & |A| = \left(A^\ast A\right)^{\frac{1}{2}}.
\end{eqnarray*}
\begin{lemma}\label{prva}
	Let $m \in \NN$ and $X, Y, Z \in \mm$. We have
	\begin{equation}\label{cvs}
		\cos \varphi \, X + \sin \varphi \, Y + Z \in \un
	\end{equation}
	for all $\varphi \in \RR$ if and only if
	\begin{eqnarray*}
		|X|^2 = |Y|^2 = I_m - |Z|^2 & {\rm and} & \re \left(X^\ast Y\right) = \re \left(X^\ast Z\right) = \re \left(Y^\ast Z\right) = 0_m.
	\end{eqnarray*}
\end{lemma}

\begin{proof}
	For any $\varphi \in \RR$ we have \eqref{cvs} if and only if
	\begin{eqnarray*}
		I_m = |\cos \varphi \, X + \sin \varphi \, Y + Z|^2 =
		|Z|^2 + |Y|^2 + \cos^2 \varphi \left(|X|^2 - |Y|^2\right) + \\
		2 \cos \varphi \sin \varphi \re \left(X^\ast Y\right) + 2 \cos \varphi \re \left(X^\ast Z\right) + 2 \sin \varphi \re \left(Y^\ast Z\right).
	\end{eqnarray*}
\end{proof}


By applying Lemma \ref{prva} for $Z = 0_m$ and unitary $X$ and $Y$, we get the following.

\begin{cor}\label{phk}
	Let $m, n \in \NN$, $\mathcal M \subset \mm$ a real vector subspace, and $\phi \colon \mathcal M \to \mm[n]$ a real linear map such that $\phi(\mathcal M \cap \un) \subset \un[n]$. Then
	\begin{eqnarray*}
		\re \left(U^\ast V\right) = 0_m \implies \re \left(\phi(U)^\ast \phi(V)\right) = 0_n, & U, V \in \mathcal M \cap \un.
	\end{eqnarray*}
\end{cor}

\begin{lemma}\label{abh}
	Let $m, n \in \NN$. If $A, B \in \mm[2n,2m]$ are such that $A\left(H \otimes I_m\right)=\left(H \otimes I_n\right)B$ for all $H \in \ihz[2]$, then $A=B=I_2 \otimes T$ for some $T \in \mm[m,n]$.
\end{lemma}

\begin{proof}
	Note that $B = \left(H \otimes I_n\right)A\left(H \otimes I_m\right)$ is constant when we vary $H \in \ihz[2]$. By choosing $H = \begin{bmatrix}
		& e^{it} \\
		e^{-it} &
	\end{bmatrix}$ for all $t \in \RR$ we deduce that $A = T \oplus T'$ for some $T, T' \in \mm[m,n]$. By choosing $H = \diag(1,-1)$ we obtain $T'=T$.
\end{proof}

\begin{prop}\label{h0u}
	Let $m \in \NN$ and let $\phi \colon \hz[2] \to \mm$ be a map. Then the following two statements are equivalent:
	\begin{itemize}
		\item $\phi$ is real linear and $\phi (\ihz[2]) \subset \un$.
		\item There exist $U, V \in \un$ and $n \in \NN$ such that $m = 2n$ and
		\begin{eqnarray*}
			\phi(A) = U \left( A \otimes I_n \right) V, & A \in \hz[2].
		\end{eqnarray*}
	\end{itemize}
\end{prop}

\begin{proof}
	The second statement clearly implies the first one.
	
	Assume now that the first statement holds. Denote
	$$
	U_0 = \phi \left(\diag(1,-1)\right) \in \un.
	$$
	Corollary \ref{phk} implies that
	\begin{eqnarray*}
		\phi \left(\begin{bmatrix}
			0 & 1 \\
			1 & 0
		\end{bmatrix}\right) = iU_0H & {\rm and} & \phi \left(\begin{bmatrix}
			0 & i \\
			-i & 0
		\end{bmatrix}\right) = iU_0K
	\end{eqnarray*}
	for some $H, K \in \ih$ such that $H \circ K = 0_m$. Hence, $m = 2n$ for some $n \in \NN$ and there exists $V \in \un$ such that
	\begin{eqnarray*}
		H = V^\ast \begin{bmatrix}
			& -iI_n \\
			iI_n &
		\end{bmatrix} V
		& {\rm and} &
		K = V^\ast \begin{bmatrix}
			& I_n \\
			I_n &
		\end{bmatrix} V.
	\end{eqnarray*}
	The desired conclusion holds for $U = U_0 V^\ast \left(I_n \oplus \left(-I_n\right)\right) \in \un$.
\end{proof}

\begin{cor}\label{c2m}
	Let $m \in \NN$ and let $\phi \colon \hz[2] \to \HH_m$ be a map. Then the following two statements are equivalent:
	\begin{itemize}
		\item $\phi$ is linear and $\phi (\ihz[2]) \subset \ih$.
		\item There exist $p, q \in \NN \cup \{0\}$ and $U \in \un$ such that $m = 2(p+q)$ and
		\begin{eqnarray*}
			\phi(A) = U \left( \left(A \otimes I_p\right) \oplus \left(-A \otimes I_q\right) \right) U^\ast, & A \in \hz[2].
		\end{eqnarray*}
	\end{itemize}
\end{cor}

\begin{proof}
	The second statement clearly implies the first one.
	
	Assume now that the first statement holds. By Proposition \ref{h0u}, there exist $U, V \in \un$ and $n \in \NN$ such that $m = 2n$ and
	\begin{eqnarray*}
		\phi(A) = U \left( A \otimes I_n \right) V, & A \in \hz[2].
	\end{eqnarray*}
	After composing $\phi$ with a unitary similarity, we may assume that $V = I_m$. Then
	\begin{eqnarray*}
		U \left( A \otimes I_n \right) \in \HH_m, & A \in \hz[2].
	\end{eqnarray*}
	Lemma \ref{abh} now yields that $U = I_2 \otimes H$ for some $H \in \ih[n]$. Hence,
	\begin{eqnarray*}
		\phi(A) = A \otimes H, & A \in \hz[2].
	\end{eqnarray*}
	After composing $\phi$ with a unitary similarity, we may assume that $H = I_p \oplus \left(-I_q\right)$ for some $p, q \in \NN \cup \{0\}$ such that $p+q=n$. In order to complete the proof, we compose $\phi$ with a permutation similarity.
\end{proof}

\begin{cor}\label{nuh}
	There does not exist a real linear map $\phi \colon \mm[2] \to \HH_2$ such that $\phi \left(\un[2]\right) \subset \ih[2]$.
\end{cor}

\begin{proof}
	Suppose that such $\phi$ exists. By Corollary \ref{c2m}, after composing $\phi$ with a unitary similarity and possibly multiplying it by $-1$, we may assume that
	\begin{eqnarray*}
		\phi(A) = A, & A \in \hz[2].
	\end{eqnarray*}
	Corollary \ref{phk} implies that
	\begin{eqnarray*}
		\phi(iI_2) \circ H = 0_2, & H \in \ihz[2].
	\end{eqnarray*}
	Lemma \ref{abh} yields $0_2 = \phi(iI_2) \in \ih[2]$, a contradiction.
\end{proof}

\begin{lemma}\label{xyuv}
	Let $m \in \NN$ and $X \in \mm$. If there exists $Y \in \mm$ such that
	\begin{equation}\label{xpy}
		X \pm Y \in \un,
	\end{equation}
	then there exist $U, V \in \un$, $p, q, r \in \NN \cup \{0\}$, $m_1, \ldots, m_r \in \NN$, and $0 < s_r < \ldots < s_1 < 1$ such that $p+q+\sum_{j=1}^r m_j = m$ and
	$$
	X = U \left(I_p \oplus 0_q \oplus \bigg(\bigoplus_{j=1}^r s_j I_{m_j}\bigg)\right) V.
	$$
	If $Y \in \mm$ satisfies \eqref{xpy}, then there exist $W \in \un[q]$ and $H_j \in \ih[m_j]$, $j = 1, \ldots, r$, such that
	$$
	Y = U \left(0_p \oplus W \oplus \bigg(\bigoplus_{j=1}^r i \sqrt{1 - s_j^2} H_j\bigg)\right) V.
	$$
\end{lemma}

\begin{proof}
	Let $Y \in \mm$ satisfy \eqref{xpy}. Then $X^\ast X + Y^\ast Y = I_m$ and $\re \left(X^\ast Y\right) = 0_m$. Hence, $\|X\| \le 1$, so $X$ has the above described singular value decomposition. Replace the pair $(X,Y)$ by the pair $(U^\ast X V^\ast, U^\ast Y V^\ast)$ and define
	\begin{eqnarray*}
		D = \bigoplus_{j=1}^r s_j I_{m_j} & {\rm and} & \Sigma = I_p \oplus 0_q \oplus D.
	\end{eqnarray*}
	Then $X = \Sigma$ and consequently, $\left|Y\right| = \left(I_m - \Sigma^2\right)^{\frac{1}{2}}$. Thus, if $T \in \un$ is from the polar decomposition $Y = T |Y| = T \left(I_m - \Sigma^2\right)^{\frac{1}{2}}$, then
	$$
	\re \left(\Sigma T \left(I_m - \Sigma^2\right)^{\frac{1}{2}}\right) = \re \left(X^\ast Y\right) = 0_m,
	$$
	from where we easily deduce that $T = Z' \oplus W \oplus Z$ for some $Z' \in \un[p]$, $W \in \un[q]$, and $Z \in \un[m-p-q]$ such that
	$$
	\re \left(D Z \left(I_{m-p-q} - D^2\right)^{\frac{1}{2}}\right) = 0_{m-p-q}.
	$$
	Hence, for $E = \left(I_{m-p-q} - D^2\right)^{\frac{1}{2}} D^{-1}$ we have
	\begin{equation}\label{reze}
		\re \left(ZE\right) = D^{-1} \re \left(D Z \left(I_{m-p-q} - D^2\right)^{\frac{1}{2}}\right) D^{-1} = 0_{m-p-q}.
	\end{equation}
	Consequently,
	\begin{equation}\label{erez}
		E \circ \re Z = \re \left(ZE\right) + Z^\ast \re \left(ZE\right)Z = 0_{m-p-q}.
	\end{equation}
	Lemma \ref{abinv} now implies that $\re Z = \bigoplus_{j=1}^r K_j$ for some $K_j \in \HH_{m_j}$, $j = 1, \ldots, r$. By applying \eqref{erez}, we obtain $K_j = 0_{m_j}$, $j = 1, \ldots, r$. Thus, $Z^\ast = -Z$. It now follows from \eqref{reze} that $Z$ commutes with $E$, which yields the required result.
\end{proof}

\begin{lemma}\label{upv}
	Let $k, m \in \NN$, $k \ge 2$, and let $\phi \colon \hz[2k] \to \mm$ be a real linear map such that
	\begin{equation}\label{phiku}
		\phi(\ihz[2k]) \subset \un.
	\end{equation}
	Then there exist $U, V \in \un$, $W \in \un[2k]$, $p, q, r \in \NN \cup \{0\}$, $n_1, \ldots, n_r \in \NN$, and real linear maps $\psi \colon \hz[2(k-1)] \to \mm[p]$ and $\phi_j \colon \hz[2(k-1)] \to \HH_{n_j}$ such that $q+r>0$, $p+2q+2\sum_{j=1}^r n_j = m$, $\psi \left(\ihz[2(k-1)]\right) \subset \un[p]$, $\phi_j \left(\ihz[2(k-1)]\right) \subset \ih[n_j]$, $j = 1, \ldots, r$, and for all $A \in \hz[2(k-1)]$, $B \in \hz[2]$ we have
	$$
	\phi \left(W^\ast \left(A \oplus B\right) W \right) = U \left(\psi(A) \oplus \left(B \otimes I_q\right) \oplus \bigoplus_{j=1}^r \left(s_j B \otimes I_{n_j} + i \sqrt{1-s_j^2} I_2 \otimes \phi_j(A)\right)\right) V.
	$$
\end{lemma}

\begin{proof}
	It follows from \eqref{phiku} that $\phi \ne 0$. Since the real vector space $\hz[2k]$ can be generated by matrices which are unitarily similar to $0_{2(k-1)} \oplus \diag(1,-1)$, there exists a matrix of this type which is not in the kernel of $\phi$. After replacing $\phi$ with the map $A \mapsto \phi\left(W^\ast A W\right)$ for a proper $W \in \un[2k]$, we may assume that
	\begin{equation}\label{pqk}
		\phi \left(0_{2(k-1)} \oplus \diag(1,-1)\right) \ne 0_m.
	\end{equation}
	
	For all $H \in \ihz[2(k-1)]$ and $K \in \ihz[2]$ we have
	$$
	\phi \left(0_{2(k-1)} \oplus K\right) \pm \phi \left(H \oplus 0_2\right) \in \un.
	$$
	We first fix $H$ and apply Lemma \ref{xyuv} to deduce that there exist $p, q', r \in \NN \cup \{0\}$, $0 < s_r < \ldots < s_1 < 1$, $m_1, \ldots, m_r \in \NN$ such that $p + q' + \sum_{j=1}^r m_j = m$ and after composing $\phi$ with multiplications by unitaries from each side we may assume that for each $K \in \ihz[2]$ there exist $W \in \un[q']$ and $W_j \in \un[m_j]$, $j = 1, \ldots, r$, such that
	$$
	\phi \left(0_{2(k-1)} \oplus K\right) = 0_p \oplus W \oplus \left(\bigoplus_{j=1}^r s_j W_j\right).
	$$
	By Proposition \ref{h0u}, there exist $q \in \NN \cup \{0\}$ and $n_j \in \NN$ such that $q' = 2q$ and $m_j = 2n_j$, $j = 1, \ldots, r$, and after composing $\phi$ with multiplications by unitaries from each side we may assume that
	\begin{eqnarray*}
		\phi \left(0_{2(k-1)} \oplus B\right) = 0_p \oplus \left(B \otimes I_q\right) \oplus \left(\bigoplus_{j=1}^r s_j B \otimes I_{n_j}\right), & B \in \hz[2].
	\end{eqnarray*}
	It follows from \eqref{pqk} that $q+r>0$.
	
	We apply Lemma \ref{xyuv} again to deduce that for each $H \in \ihz[2(k-1)]$ there exist $Z \in \un[p]$ and $Z_j \in \un[m_j]$, $j = 1, \ldots, r$, such that
	$$
		\phi \left(H \oplus 0_2\right) = Z \oplus 0_{2q} \oplus \left(\bigoplus_{j=1}^r \sqrt{1-s_j^2} Z_j\right)
	$$
	and
	\begin{eqnarray*}
		\re \left(\left(K \otimes I_{n_j}\right) Z_j\right) = 0_{m_j}, & K \in \ihz[2].
	\end{eqnarray*}
	By Lemma \ref{abh}, there exists $K_j \in \ih[n_j]$ such that $Z_j = i I_2 \otimes K_j$.
\end{proof}

\begin{lemma}\label{toto}
	Let $k, m \in \NN$, $k \ge 2$, and let $\phi \colon \hz[2k] \to \mm$ be a real linear map such that
	$$
	\phi(\ihz[2k]) \subset \un.
	$$
	If $H \in \ihz[2(k-2)]$, $K \in \ih[2]$, $Z \in \un[2]$, $T \in \mm[2(k-1),2]$, and $V \in \mm[2(k-1),2(k-2)]$ are such that $\begin{bmatrix}
		V & T
	\end{bmatrix} \in \un[2(k-1)]$, then
	$$
		\left|\phi \left(\begin{bmatrix}
			& TZ \\
			Z^\ast T^\ast &
		\end{bmatrix}\right)\right|^2 = \left|\phi \left(TKT^\ast \oplus \left(-Z^\ast KZ\right)\right)\right|^2 = I_m - \left|\phi \left(VHV^\ast \oplus 0_2\right)\right|^2
	$$
	and
	$$
		\re \left(\phi \left(TKT^\ast \oplus \left(-Z^\ast KZ\right)\right)^\ast \phi \left(\begin{bmatrix}
			& TZ \\
			Z^\ast T^\ast &
		\end{bmatrix}\right)\right) = 0_m.
	$$
\end{lemma}

\begin{proof}
	Let $U = \begin{bmatrix}
		V & T
	\end{bmatrix} \in \un[2(k-1)]$. For all $\varphi \in \RR$ we have
	\begin{eqnarray*}
		\phi \left(VHV^\ast \oplus 0_2\right) + \cos \varphi \, \phi \left(TKT^\ast \oplus \left(-Z^\ast KZ\right)\right) + \sin \varphi \, \phi \left(\begin{bmatrix}
			& TZ \\
			Z^\ast T^\ast &
		\end{bmatrix}\right) = \\
		\phi \left( \left(U \oplus I_2\right)
		\left( H \oplus \begin{bmatrix}
			\cos \varphi \, K & \sin \varphi \, Z \\
			\sin \varphi \, Z^\ast & -\cos \varphi \, Z^\ast KZ
		\end{bmatrix}\right)
		\left(U^\ast \oplus I_2\right) \right) \in \un.
	\end{eqnarray*}
	Lemma \ref{prva} yields the required result.
\end{proof}

\begin{lemma}\label{kq}
	Let $k, q \in \NN$, $k \ge 2$, and let $\phi \colon \hz[2k] \to \mm[2q]$ be a real linear map such that
	$$
	\phi(\ihz[2k]) \subset \un[2q]
	$$
	and there exist $U, V \in \un[2q]$ such that
	\begin{eqnarray*}
		\phi \left(A \oplus B\right) = U \left(B \otimes I_q\right) V, & A \in \hz[2(k-1)], B \in \hz[2].
	\end{eqnarray*}
	Then there exists a real linear map $\psi \colon \mm[2(k-1)] \to \HH_q$ such that $\psi \left(\un[2(k-1)]\right) \subset \ih[q]$.
\end{lemma}

\begin{proof}
	Let $W \in \un[2(k-1)]$ and $T_0 \in \mm[2(k-1),2]$ such that $T_0^\ast T_0 = I_2$. Let $T = W T_0$. By Lemma \ref{toto},
	$\phi \left(\begin{bmatrix}
		& T \\
		T^\ast &
	\end{bmatrix}\right) \in \un[2q]$ and
	\begin{eqnarray*}
		\re \left(V^\ast \left(K \otimes I_q\right) U^\ast \phi \left(\begin{bmatrix}
			& T \\
			T^\ast &
		\end{bmatrix}\right)\right) = 0_{2q}, & K \in \ihz[2].
	\end{eqnarray*}
	Lemma \ref{abh} now implies that $U^\ast \phi \left(\begin{bmatrix}
		& T \\
		T^\ast &
	\end{bmatrix}\right) V^\ast = i I_2 \otimes L$ for some $L \in \ih[q]$.
\end{proof}

\begin{prop}\label{hura}
	Let $k, m \in \NN$ be such that there exists a real linear map $\phi \colon \hz[2k] \to \mm$ such that
	$$
	\phi(\ihz[2k]) \subset \un.
	$$
	Then $m \ge 2k$.
\end{prop}

\begin{proof}
	We apply induction on $k$. The required statement for $k = 1$ follows from Proposition \ref{h0u}.
	
	Assume now that $k \ge 2$ and that the statement holds for $k-1$. Let $p, q, r$ and $n_j$, $j = 1, \ldots, r$, be as in Lemma \ref{upv}. Since $q+r>0$, we distinguish three cases.
	
	\begin{case}{$r > 0$}\end{case}
	If $j \in \{1, \ldots, r\}$, then by Lemma \ref{upv} and the induction hypothesis we have
	\begin{equation*}
		m \ge 2n_j \ge 4(k-1) \ge 2k.
	\end{equation*}
	
	\begin{case}{$r = p = 0$, $q > 0$}\end{case}
	Lemma \ref{upv}, Lemma \ref{kq}, and the induction hypothesis yield $q \ge 2(k-1)$. Hence,
	\begin{equation*}
		m = 2q \ge 4(k-1) \ge 2k.
	\end{equation*}
	
	\begin{case}{$r = 0$, $p, q > 0$}\end{case}
	Lemma \ref{upv} and the induction hypothesis imply that
	\begin{equation*}
		m = p+2q \ge 2(k-1+q) \ge 2k.
	\end{equation*}
\end{proof}

\begin{rem}
	Example \ref{e2km} shows that maps in Proposition \ref{hura} need not be injective.
\end{rem}

\begin{theorem}\label{iho}
	Let $k \in \NN$ and let $\phi \colon \hz[2k] \to \hz[2k]$ be a map.
	Then the following two statements are equivalent:
	\begin{itemize}
		\item $\phi$ is linear and $\phi(\ihz[2k]) \subset \ihz[2k]$.
		\item There exist $U \in \un[2k]$ and $s \in \{-1,1\}$ such that we have either
		\begin{eqnarray*}
			\phi(A) = s UAU^\ast, & A \in \hz[2k],
		\end{eqnarray*}
		or $k \ge 2$ and
		\begin{eqnarray*}
			\phi(A) = s U\overline{A}U^\ast, & A \in \hz[2k].
		\end{eqnarray*}
	\end{itemize}
\end{theorem}

\begin{proof}
	One implication is clear, so assume that the first statement holds.
	
	We apply induction on $k$. The desired conclusion in the case $k = 1$ follows form Corollary \ref{c2m}.
	
	Assume now that $k \ge 2$ and that the statement holds for $k-1$. Let $U, V, W \in \un[2k]$, $p, q, r$ and $n_j$, $j = 1, \ldots, r$, be as in Lemma \ref{upv}. After replacing $\phi$ with the map $A \mapsto V \phi \left(W^\ast AW\right) V^\ast$, we may assume that $V = W = I_{2k}$.
	
	Seeking a contradiction, suppose that $r > 0$. By Lemma \ref{upv} and Proposition \ref{hura}, $n_j \ge 2(k-1)$, $j = 1, \ldots, r$. Since $n_j \le k$, we obtain that $k=n_1=2$, $r=1$, and $p=q=0$. By applying Lemma \ref{upv} and Corollary \ref{c2m} we deduce that after composing $\phi$ with a unitary similarity and changing $U$, we may assume that there exists $\varphi_0 \in \left(-\tfrac{\pi}{2},\tfrac{\pi}{2}\right) \setminus \{0\}$ such that for all $H, K \in \hz[2]$ we have
	$$
	U \left(\cos \varphi_0 \, K \otimes I_2 + i \sin \varphi_0 \, I_2 \otimes H \right) = \phi \left(H \oplus K\right) \in \hz[4].
	$$
	Choosing $H = 0_2$ implies that $U \left(K \otimes I_2\right) \in \hz[4]$. By Lemma \ref{abh}, $U = I_2 \otimes H'$ for some $H' \in \ih[2]$ such that $H' \circ H = 0_2$ for all $H \in \ihz[2]$. Another application of Lemma \ref{abh} yields $H' = 0_2$, a contradiction. Thus, $r = 0$ and $q = r+q > 0$.

Seeking another contradiction, assume that $p = 0$. Lemma \ref{upv}, Lemma \ref{kq}, and Proposition \ref{hura} yield $k \ge 2(k-1)$. Therefore, $k = 2$. By Lemma \ref{kq}, there exists a real linear map $\mm[2] \to \HH_2$ which maps $\un[2]$ into $\ih[2]$, contradicting Corollary \ref{nuh}. Thus, $p > 0$.

By Lemma \ref{upv} and Proposition \ref{hura}, $p \ge 2(k-1)$. Since $q > 0$ and $p + 2q = 2k$, we have $q = 1$ and $p = 2(k-1)$. Let $\psi$ be as in Lemma \ref{upv}. The induction hypothesis implies that there exists $s \in \{-1,1\}$ such that after again composing $\phi$ with a unitary similarity and changing $U$, we may assume that we have either
\begin{eqnarray*}
	\psi (A) = sA, & A \in \hz[2(k-1)],
\end{eqnarray*}
or $k \ge 3$ and
\begin{eqnarray*}
	\psi (A) = s\overline{A}, & A \in \hz[2(k-1)].
\end{eqnarray*}
If the second possibility occurs, then we apply Example \ref{lk2k} to conclude that after composing $\phi$ by entry-wise complex conjugation, multiplying it by $-1$, composing it with another unitary similarity and changing $U$, we may assume that
\begin{eqnarray*}
	\phi \left(A \oplus B\right) = U \left(\left(sA\right) \oplus B\right), & A \in \hz[2(k-1)], \, B \in \hz[2].
\end{eqnarray*}
Thus,
\begin{eqnarray*}
	U \left(A \oplus B\right) \in \hz[2k], & A \in \hz[2(k-1)], B \in \hz[2].
\end{eqnarray*}
We can apply Lemma \ref{abh} to deduce that $U = U' \oplus \left(\pm I_2\right)$ for some $U' \in \un[2(k-1)]$ such that $U' A \in \hz[2(k-1)]$ for all $A \in \hz[2(k-1)]$. We can choose $A = 0_{2(j-1)} \oplus H \oplus 0_{2(k-1-j)}$ for all $j \in \{1, \ldots, k-1\}$ and $H \in \ihz[2]$ and apply Lemma \ref{abh} again to show that $U' = \pm I_{2(k-1)}$. Hence, after possibly multiplying $\phi$ by $-1$ again, we may assume that
\begin{eqnarray*}
	\phi \left(A \oplus B\right) = \left(sA\right) \oplus B, & A \in \hz[2(k-1)], \, B \in \hz[2].
\end{eqnarray*}

	Define the real linear map $\tau \colon \mm[2(k-1),2] \to \hz[2k]$ by
	\begin{eqnarray*}
		\tau(T) = \phi \left(\begin{bmatrix}
				& T \\
				T^\ast &
			\end{bmatrix}\right), & T \in \mm[2(k-1),2].
	\end{eqnarray*}
	Let $T \in \mm[2(k-1),2]$ such that $T^\ast T = I_2$. Lemma \ref{toto} implies that
	\begin{equation}\label{tat}
		\tau(T)^2 = \left(TT^\ast\right) \oplus I_2
	\end{equation}
	and
	\begin{eqnarray}\label{stk}
		\left(\left(sTKT^\ast\right) \oplus \left(-K\right)\right) \circ \tau(T) = 0_{2k}, & K \in \ihz[2].
	\end{eqnarray}
	Denote $\tau(T) = \begin{bmatrix}
		A' & T' \\
		T'^\ast & B'
	\end{bmatrix}$ for some $A' \in \HH_{2(k-1)}$, $B' \in \HH_2$, $T' \in \mm[2(k-1),2]$. It follows from \eqref{stk} that $K \circ B' = 0_2$, $K \in \ihz[2]$. By Lemma \ref{abh}, $B' = 0_2$. We now deduce from \eqref{tat} that $A'^2 + T'T'^\ast = TT^\ast$ and $T'^\ast T' = I_2$. Hence, $\tr A'^2 = 0$, implying that $A' = 0_{2(k-1)}$. We obtain from \eqref{stk} that $s TKT^\ast T' = T'K$, $K \in \ihz[2]$. Consequently, $s KT^\ast T' = T^\ast T'K$, $K \in \ihz[2]$. By Lemma \ref{abh}, $s = 1$ and $T^\ast T' = z I_2$ for some $z \in \CC$. This and $T'T'^\ast = TT^\ast$ imply that $|z|=1$ and $T' = zT$. We have shown that
	\begin{eqnarray*}
		\phi \left(A \oplus B\right) = A \oplus B, & A \in \hz[2(k-1)], \, B \in \hz[2],
	\end{eqnarray*}
	and there exists $z_T \in \CC$ such that $|z_T| = 1$ and
	$$
		\tau(T) = \begin{bmatrix}
		& z_T T \\
		\overline{z_T} T^\ast &
	\end{bmatrix}.
	$$
	Let $L \in \ih[2(k-1)]$. By calculating $\tau \left(\tfrac{\sqrt{2}}{2} \left(T + i LT\right)\right)$ we obtain the existence of $z' \in \CC$ such that $|z'| = 1$ and
	$$
	z_T T + i z_{iLT} LT = z' \left(T + i LT\right).
	$$
	If $LT \ne \pm T$, then $LT$ is not a complex scalar multiple of $T$, so $z_{iLT} = z_T$, while the choice $L = I_{2(k-1)}$ yields
	\begin{equation}\label{zz}
		z_{iT} = \pm z_T.
	\end{equation}
	
	Fix now $T_0 = \begin{bmatrix}
		I_2 \\ 0_{2(k-2),2}
	\end{bmatrix}$. After composing $\phi$ with a unitary similarity we may assume that
	\begin{equation}\label{tto}
		z_{T_0} = z_{iLT_0} = 1
	\end{equation}
	for every $L \in \ih[2(k-1)]$ for which $LT_0 \ne \pm T_0$. By applying \eqref{zz}, we deduce that there exists $s_0 \in \{-1,1\}$ such that
	\begin{equation}\label{tito}
		z_{iT_0} = z_{LT_0} = s_0
	\end{equation}
	for every $L \in \ih[2(k-1)]$ for which $LT_0 \ne \pm T_0$.
	
	Let $P \in \mathcal P_{2(k-1),2}$. Then there exists $T \in \mm[2(k-1),2]$ such that $T^\ast T = I_2$ and $P=TT^\ast$. Let $Z \in \un[2]$. We apply Lemma \ref{toto} to show that
	\begin{equation}\label{ppi}
		\phi \left(P \oplus \left(-I_2\right) \right)^2 = P \oplus I_2
	\end{equation}
	and
	\begin{equation}\label{pzt}
	\phi \left(P \oplus \left(-I_2\right) \right) \circ \tau(TZ) = 0_{2k}.
	\end{equation}
	We claim that
	\begin{eqnarray}\label{tojeto}
		\phi \left(P \oplus \left(-I_2\right)\right) \circ \begin{bmatrix}
			& TZ \\ Z^\ast T^\ast &
		\end{bmatrix} = 0_{2k}, & Z \in \un[2].
	\end{eqnarray}
	Indeed, if $k = 2$, then
	$T_0 = I_2$ and $TZ \in \un[2]$ is a real linear combination of $I_2$, $iI_2$, $L$, and $iM$ for some $L, M \in \ihz[2]$, so \eqref{tto}, \eqref{tito}, \eqref{pzt} yield \eqref{tojeto}. On the other hand, if $k \ge 3$, then there exist $L, M \in \ihz[2(k-1)]$ such that $iLT_0 = MT_0 \ne \pm T_0$, so we obtain from \eqref{tto} and \eqref{tito} that $s_0 = 1$, implying that
	\begin{eqnarray*}
		\tau(Y) = \begin{bmatrix}
			& Y \\
			Y^\ast &
		\end{bmatrix}, & Y \in \mm[2(k-1),2].
	\end{eqnarray*}
	Again, \eqref{tojeto} follows from \eqref{pzt}.
	
	Assume again that $k \ge 2$. We assert that there exists $s_P \in \{-1,1\}$ such that
	\begin{equation}\label{phpi}
		\phi \left(P \oplus \left(-I_2\right) \right) = s_P \left(P \oplus \left(-I_2\right) \right).
	\end{equation}
	Denote $\phi \left(P \oplus \left(-I_2\right) \right) = \begin{bmatrix}
		A' & T' \\
		T'^\ast & B'
	\end{bmatrix}$ for some $A' \in \HH_{2(k-1)}$, $B' \in \HH_2$, $T' \in \mm[2(k-1),2]$. It follows from \eqref{tojeto} that $\re \left(T' Z^\ast T^\ast\right) = 0_{2(k-1)}$. After replacing $Z$ by $iZ$ we obtain $T' Z^\ast T^\ast = 0_{2(k-1)}$, yielding that $T' = 0_{2(k-1),2}$. We now deduce from \eqref{ppi} that $A'^2 = P$ and $B' \in \ih[2]$. By applying \eqref{tojeto} again we get $A'TZ + TZB' = 0_{2(k-1),2}$, so $T^\ast A'TZ + ZB' = 0_2$. By Lemma \ref{abh}, there exists $s_P \in \{-1,1\}$ such that $B' = -s_P I_2$. Consequently, $A'T = s_P T$, yielding that $A'P = s_P P$. Hence, $A'$ and $s_P P$ coincide on $\ran P$ and also on $\ker P = \ker A'^2 = \ker A'$. Therefore, $A' = s_P P$.
	
	We now distinguish two cases. First, suppose that $k = 2$. Then $P = T_0 = I_2$. Denote $t_0 = s_{I_2}$. It follows from
		$$
		\tfrac{1}{2}\begin{bmatrix}
				t_0 I_2-H & I_2+s_0H \\
				I_2+s_0H & H-t_0 I_2
			\end{bmatrix} = \phi \left(\tfrac{1}{2}\begin{bmatrix}
				I_2-H & I_2+H \\
				I_2+H & H-I_2
			\end{bmatrix}\right) \in \ihz[4]
		$$
		that $t_0 = s_0$. Thus, if $s_0=1$, then $\phi(A) = A$, $A \in \hz[4]$. Assume now that $s_0 = -1$. Then one can easily calculate by using Example \ref{lk2k} that
		\begin{eqnarray*}
			\phi(A) = -V_0 \overline{A} V_0^\ast, & A \in \hz[4],
		\end{eqnarray*}
		where $V_0 = U_0 \oplus \left(-U_0\right) \in \un[4]$ for $U_0 = \begin{bmatrix}
				0 & 1 \\ -1 & 0
			\end{bmatrix} \in \un[2]$.
		
		Suppose now that $k \ge 3$. Then there exists $Q \in \mathcal P_{2(k-1),2} \setminus \{P\}$. Thus,
	\begin{eqnarray*}
		\left(P-Q\right) \oplus 0_2 = 
		\phi \left(\left(P-Q\right) \oplus 0_2\right) = \\
		\phi \left(P \oplus \left(-I_2\right)\right) - \phi \left(Q \oplus \left(-I_2\right)\right) = \left(s_PP - s_QQ\right) \oplus \left(\left(s_Q - s_P\right) I_2\right).
	\end{eqnarray*}
	Hence, $s_P = 1$. Consequently,
	\begin{eqnarray*}
		\phi \left(A \oplus \left(-\tfrac{\tr A}{2} I_2\right) \right) = A \oplus \left(-\tfrac{\tr A}{2} I_2\right), & A \in \HH_{2(k-1)},
	\end{eqnarray*}
	implying that $\phi(A) = A$, $A \in \hz[2k]$.
\end{proof}

\section{Linear maps sending projections of a fixed finite rank to projections of a finite rank}

In this section, we will study linear maps sending projections of a fixed finite rank to projections of a finite rank. We will make a complete description of such maps in the case when the first Hilbert space is $2$-dimensional.

We start by an observation that linear maps sending projections of a fixed finite rank to projections of a finite rank must send projections of a fixed finite rank to projections of a fixed rank.

\begin{prop}\label{fkm}
	Let $\HH$ and $\KK$ be complex Hilbert spaces and $k \in \NN$, $k < \dim \HH$. Let $\phi \colon \fs \to \fs[\KK]$ be a linear map such that
	$$
	\phi \left(\PH k\right) \subset \PH[\KK]{f}.
	$$
	There exists $m \in \NN \cup \{0\}$, $m \le \dim \KK$, such that $\phi \left(\PH k\right) \subset \PH[\KK]{m}$.
\end{prop}

\begin{proof}
	Denote by $f$ the induced map $\G k \to \G[\KK]{f}$, i.e.
	\begin{eqnarray*}
		\phi \left(P_X\right) = P_{f(X)}, & X \in \G k.
	\end{eqnarray*}
	Let $X, Y \in \G k$. We need to show that $\dim f(X) = \dim f(Y)$.
	
	Assume first that $\dg{X}{Y} < 1$. Then there exists $a \in (1, \infty)$ such that $\dg{X}{Y} \le \tfrac{\sqrt{2a-1}}{a}$. Proposition \ref{prophol} implies that $\hol[a]{X}{Y} \ne \emptyset$. Corollary \ref{dimhol} yields $\hol[a]{X}{Y} \subset \G k$. Since $\hol[a]{f(X)}{f(Y)}$ contains $f \left( \hol[a]{X}{Y} \right)$, it is not empty. We apply Corollary \ref{dimhol} again to obtain that $\dim f(X) = \dim f(Y)$.
	
	Suppose now that $\dg{X}{Y} = 1$. By Proposition \ref{prophol}, there exists an orthonormal basis $\mathcal B$ of $X+Y$ such that
	$$
	(P_X|_{X+Y})_{\mathcal B} = I_m \oplus I_p \oplus 0_p \oplus \bigoplus_{j=1}^r \begin{bmatrix}
		\tfrac{1+\cos \varphi_j}{2} I_{m_j}  & \tfrac{\sin \varphi_j}{2} I_{m_j} \\
		\tfrac{\sin \varphi_j}{2} I_{m_j} & \tfrac{1-\cos \varphi_j}{2} I_{m_j}
	\end{bmatrix}
	$$
	and
	$$
	(P_Y|_{X+Y})_{\mathcal B} = I_m \oplus 0_p \oplus I_p \oplus \bigoplus_{j=1}^r \begin{bmatrix}
		\tfrac{1+\cos \varphi_j}{2} I_{m_j}  & -\tfrac{\sin \varphi_j}{2} I_{m_j} \\
		-\tfrac{\sin \varphi_j}{2} I_{m_j} & \tfrac{1-\cos \varphi_j}{2} I_{m_j}
	\end{bmatrix}
	$$
	for some $p \in \NN$, $m, r \in \NN \cup \{0\}$, $m_1, \ldots, m_r \in \NN$, and $0 < \varphi_r < \ldots < \varphi_1 < \tfrac{\pi}{2}$. Let $Z \in [X \cap Y, X+Y]$ be such that
	$$
	(P_Z|_{X+Y})_{\mathcal B} = I_m \oplus \left(\tfrac{1}{2} \begin{bmatrix}
		I_p & I_p \\
		I_p & I_p
	\end{bmatrix}\right) \oplus \bigoplus_{j=1}^r \left(I_{m_j} \oplus 0_{m_j}\right)
	$$
	Then $Z \in \G k$, $\dg{X}{Z} < 1$, and $\dg{Y}{Z} < 1$. Consequently, $\dim f(X) = \dim f(Z) = \dim f(Y)$.
\end{proof}

Next, we will show that linear maps sending projections of a fixed finite rank to projections of a finite rank are intimately connected to linear maps sending trace zero hermitian unitary matrices to unitary matrices. We start by the following obvious connection.

\begin{lemma}\label{a0p}
	Let $k \in \NN$ and $A \in \HH_{2k}$. Then
	$$
	A \in \ihz[2k] \iff \tfrac{1}{2}(I_{2k}+A) \in \mathcal P_{2k,k}.
	$$
\end{lemma}

For all $X \in \G{}$ denote
$$
\left\langle X\right] = \left\{ A \in \fs \mid \ran A \subset X \right\}.
$$
\begin{lemma}\label{phraz}
	Let $k \in \NN$, let $\KK$ be a complex Hilbert space, and let $\phi \colon \HH_{2k} \to \fs[\KK]$ be a linear map such that $\phi \left(\mathcal P_{2k,k}\right) \subset \PH[\KK]{f}$. Then there exist:
	\begin{itemize}
		\item $m, p, r \in \NN \cup \{0\}$, $\tfrac{1}{2} < t_1 < \ldots < t_r < 1$, $m_1, \ldots, m_r \in \NN$,
		\item a unital linear map $\phi_0 \colon \HH_{2k} \to \HH_{2p}$ such that $\phi_0 \left(\ihz[2k]\right) \subset \ihz[2p]$,
		\item real linear maps $\phi_j \colon \hz[2k] \to \mm[m_j]$ such that $\phi_j(\ihz[2k]) \subset \un[m_j]$, $j = 1, \ldots, r$,
		\item $Z \in \G[\KK]{f}$ and its orthonormal basis $\mathcal B$,
	\end{itemize}
	such that for all $A \in \HH_{2k}$ we have $\phi(A) \in \left\langle Z\right]$ and
	\begin{eqnarray*}
		\left(\phi (A)|_Z\right)_{\mathcal B} = \left(\tfrac{\tr A}{k} I_m\right) \oplus \phi_0(A) \oplus \\ \bigoplus_{j=1}^r \begin{bmatrix}
		t_j \tfrac{\tr A}{k} I_{m_j}  & \sqrt{t_j(1-t_j)} \, \phi_j \left(2A - \tfrac{\tr A}{k} I_{2k}\right)^\ast  \\
		\sqrt{t_j(1-t_j)} \, \phi_j \left(2A - \tfrac{\tr A}{k} I_{2k}\right)  & (1-t_j) \tfrac{\tr A}{k} I_{m_j}
	\end{bmatrix}.
	\end{eqnarray*}
\end{lemma}

\begin{rem}
	If we assumed that $\phi$ is injective, it would not significantly simplify the problem of describing all such possible maps $\phi$, as in order to fulfill this assumption, it is already enough that one of the maps $\phi_0$, $\phi_j$, $j = 1, \ldots, r$, is injective.
\end{rem}

\begin{proof}
	Let $P = I_k \oplus 0_k$, $Q = 0_k \oplus I_k$, $X = \ran \phi(P)$, $Y = \ran \phi(Q)$. By Proposition \ref{prophol}, there exists an orthonormal basis $\mathcal B$ of $X+Y$ such that
	$$
	(P_X|_{X+Y})_{\mathcal B} = I_m \oplus I_p \oplus 0_q \oplus \bigoplus_{j=1}^r \begin{bmatrix}
	t_j I_{m_j}  & \sqrt{t_j\left(1-t_j\right)} I_{m_j} \\
	\sqrt{t_j\left(1-t_j\right)} I_{m_j} & \left(1-t_j\right) I_{m_j}
	\end{bmatrix}
	$$
	and
	$$
	(P_Y|_{X+Y})_{\mathcal B} = I_m \oplus 0_p \oplus I_q \oplus \bigoplus_{j=1}^r \begin{bmatrix}
		t_j I_{m_j}  & -\sqrt{t_j\left(1-t_j\right)} I_{m_j} \\
		-\sqrt{t_j\left(1-t_j\right)} I_{m_j} & \left(1-t_j\right) I_{m_j}
	\end{bmatrix}
	$$
	for some $m, p, q, r \in \NN \cup \{0\}$, $m_1, \ldots, m_r \in \NN$, and $\tfrac{1}{2} < t_1 < \ldots < t_r < 1$. It follows from Proposition \ref{fkm} that $\dim X = \dim Y$. Hence, $p = q$. For any $R \in \mathcal P_{2k,k}$ we have $P_X+P_Y-\phi(R) \in \PH[\KK]{f}$. Proposition \ref{prophol} yields $\ran \phi(R) \in [X \cap Y,X+Y]$ and
	$$
	(\phi(R)|_{X+Y})_{\mathcal B} = I_m \oplus Q \oplus \bigoplus_{j=1}^r \begin{bmatrix}
		t_j I_{m_j}  & \sqrt{t_j\left(1-t_j\right)} U_j^\ast \\
		\sqrt{t_j\left(1-t_j\right)} U_j & \left(1-t_j\right) I_{m_j}
	\end{bmatrix}
	$$
	for some $Q \in \mathcal P_{2p}$ and $U_j \in \un[m_j]$, $j = 1, \ldots, r$. By Proposition \ref{fkm}, $Q \in \mathcal P_{2p,p}$. Thus, there exists a unital linear map $\phi_0 \colon \HH_{2k} \to \HH_{2p}$ such that $\phi_0(\mathcal P_{2k,k}) \subset \mathcal P_{2p,p}$, and for every $j \in \left\{1, \ldots, r\right\}$ there exists a real linear map $\psi_j \colon \HH_{2k} \to \mm[m_j]$ such that $\psi_j(\mathcal P_{2k,k}) \subset \un[m_j]$ and for each $A \in \HH_{2k}$ we have
	$$
	(\phi(A)|_{X+Y})_{\mathcal B} = \left(\tfrac{\tr A}{k}I_m\right) \oplus \phi_0(A) \oplus \bigoplus_{j=1}^r \begin{bmatrix}
		t_j I_{m_j}  & \sqrt{t_j\left(1-t_j\right)} \psi_j(A)^\ast \\
		\sqrt{t_j\left(1-t_j\right)} \psi_j(A) & \left(1-t_j\right) I_{m_j}
	\end{bmatrix}.
	$$
	Lemma \ref{a0p} yields $\phi_0 \left(\ihz[2k]\right) \subset \ihz[2p]$. For each $j \in \{1, \ldots, r\}$, Lemma \ref{a0p} and
	$$
	\psi_j \left(I_{2k}\right) = \psi_j \left(P+Q\right) = I_{m_j} - I_{m_j} = 0_{m_j}
	$$
	imply that we get the required result for the map $\phi_j$, given by $\phi_j(A) = \tfrac{1}{2} \psi_j(A)$, $A \in \hz[2k]$.
\end{proof}

\begin{theorem}\label{2f}
	Let $\HH$ and $\KK$ be complex Hilbert spaces and $\dim \HH = 2$. Let $\phi \colon \sh \to \fs[\KK]$ be a map. The following two statements are equivalent.
	\begin{itemize}
		\item $\phi$ is linear and $\phi \left(\PH{1}\right) \subset \PH[\KK]{f}$.
		\item One of the following statements holds:
		\begin{itemize}
			\item There exists $P_0 \in \PH[\KK]{f}$ such that
			\begin{eqnarray*}
				\phi (A) = (\tr A) P_0, & A \in \fs.
			\end{eqnarray*}
			\item There exist $n \in \NN$, $P_0, Q_0 \in \PH[\CC^n]{}$, $P_0 \ne Q_0$, and a linear isometry $U \colon \HH \otimes \CC^n \to \KK$ such that
			\begin{eqnarray*}
				\phi (A) = U \left(A \otimes P_0 + \left((\tr A) \mathcal I_\HH - A\right) \otimes Q_0 \right) U^\ast, & A \in \sh.
			\end{eqnarray*}
		\end{itemize}
	\end{itemize}
\end{theorem}

\begin{rem}
	The projections $P_0$ and $Q_0$ can be chosen to be in a canonical form, e.g. one of the forms described in Propositions \ref{halmd} and \ref{prophol}.
\end{rem}

\begin{rem}
	The resulting map $\phi$ is either constant on the set $\PH{1}$ or injective.
\end{rem}

\begin{rem}
	If we replace $U$ in the result by a conjugate-linear isometry, then one can deduce from Example \ref{lk2k} that such obtained map is also of the resulting form.
\end{rem}

\begin{rem}
	If the resulting map $\phi$ is injective, then it may not be a combination of the standard maps from Examples \ref{stu}, \ref{n-k}, \ref{con}, \ref{tenp}. Indeed, non-zero eigenvalues of $\phi \left(\mathcal I_\HH\right)$ are the same as non-zero eigenvalues of $P_0+Q_0$, which may be any number in $(0,2]$ by Proposition \ref{prophol}.
\end{rem}

\begin{proof}
	It is easy to see that the second statement implies the first one.
	
	Assume now that the first statement holds. After choosing an orthonormal basis of $\HH$, we identify $\sh$ with $\HH_2$. By Lemma \ref{phraz} there exist:
	\begin{itemize}
		\item $m, p, r \in \NN \cup \{0\}$, $\tfrac{1}{2} < t_1 < \ldots < t_r < 1$, $m_1, \ldots, m_r \in \NN$,
		\item a unital linear map $\phi_0 \colon \HH_2 \to \HH_{2p}$ such that $\phi_0 \left(\ihz[2]\right) \subset \ihz[2p]$,
		\item real linear maps $\phi_j \colon \hz[2] \to \mm[m_j]$ such that
		$\phi_j(\ihz[2]) \subset \un[m_j]$, $j = 1, \ldots, r$,
		\item $Z \in \G[\KK]{f}$ and its orthonormal basis $\mathcal B$,
	\end{itemize}
	such that for all $A \in \HH_2$ we have $\phi(A) \in \left\langle Z\right]$ and
	$$
	\left(\phi (A)|_Z\right)_{\mathcal B} = \left((\tr A) I_m\right) \oplus \phi_0(A) \oplus \bigoplus_{j=1}^r \psi_j(A),
	$$
	where
	$$
	\psi_j(A) = \begin{bmatrix}
		t_j (\tr A) I_{m_j}  & \sqrt{t_j(1-t_j)} \, \phi_j \left(2A - (\tr A) I_2\right)^\ast  \\
		\sqrt{t_j(1-t_j)} \, \phi_j \left(2A - (\tr A) I_2\right)  & (1-t_j) (\tr A) I_{m_j}
	\end{bmatrix}.
	$$
	By Corollary \ref{c2m}, there exists $q \in \{0, 1, \ldots, p\}$ such that we may assume after a change of $\mathcal B$ that
	\begin{eqnarray*}
		\phi_0(A) = \left(A \otimes I_q\right) \oplus \left(-A \otimes I_{p-q}\right), & A \in \hz[2].
	\end{eqnarray*}
	Since $\phi_0$ is unital,
	\begin{eqnarray*}
		\phi_0(A) = \left(A \otimes I_q\right) \oplus \left(\left((\tr A)I_2-A\right) \otimes I_{p-q}\right), & A \in \HH_2.
	\end{eqnarray*}
	Let now $j \in \{1, \ldots, r\}$. By Proposition \ref{h0u}, there exists $n_j \in \NN$ such that $m_j=2n_j$ and we may assume after a change of $\mathcal B$ that
	\begin{eqnarray*}
		\phi_j(A) = A \otimes I_{n_j} & A \in \hz[2].
	\end{eqnarray*}
	Hence, for all $A \in \HH_2$ we have
	\begin{eqnarray*}
	\psi_j(A) =
	\left(\begin{bmatrix}
		t_j & \sqrt{t_j(1-t_j)} \\
		\sqrt{t_j(1-t_j)} & 1-t_j
	\end{bmatrix}
	\otimes A + \right. \\
	\left. \begin{bmatrix}
		t_j & -\sqrt{t_j(1-t_j)} \\
		-\sqrt{t_j(1-t_j)} & 1-t_j
	\end{bmatrix} \otimes \left((\tr A) I_2 - A\right)\right) \otimes I_{n_j}.
	\end{eqnarray*}
	If $p=r=0$, then we get the first possible form for $\phi$. If $p+r>0$, then for $n = m+p+\sum_{j=1}^{r}m_j$ we get the other form after permuting $\mathcal B$, and choosing $P_0$ and $Q_0$ corresponding to the respective matrices
	$$
	I_m \oplus I_q \oplus 0_{p-q} \oplus \bigoplus_{j=1}^r \bigoplus_{i=1}^{n_j} \begin{bmatrix}
		t_j & \sqrt{t_j(1-t_j)} \\
		\sqrt{t_j(1-t_j)} & 1-t_j
	\end{bmatrix}
	$$
	and
	$$
	I_m \oplus 0_q \oplus I_{p-q} \oplus \bigoplus_{j=1}^r \bigoplus_{i=1}^{n_j} \begin{bmatrix}
		t_j & -\sqrt{t_j(1-t_j)} \\
		-\sqrt{t_j(1-t_j)} & 1-t_j
	\end{bmatrix}.
	$$
\end{proof}

\section{Linear preservers of rank $k$ projections}

In this section, we will describe linear maps which send projections of a fixed finite rank to projections of the same or smaller rank.

\begin{theorem}\label{phkk}
	Let $\HH$ and $\KK$ be complex Hilbert spaces, $k \in \NN$, $k < \dim \HH$, $m \in \{0, 1, \ldots, k\}$, and let $\phi \colon \fs \to \fs[\KK]$ be a map. Then the following two statements are equivalent.
	\begin{itemize}
			\item $\phi$ is linear and
			\begin{equation}\label{pphk}
				\phi (\PH{k}) \subset \PH[\KK]{m}.
			\end{equation}
			\item One of the following statements holds:
			\begin{itemize}
					\item There exists $P_0 \in \PH[\KK]{m}$ such that
				\begin{eqnarray}\label{pcon}
						\phi (A) = \tfrac{\tr A}{k} P_0, & A \in \fs.
					\end{eqnarray}
				\item $m = k$ and there exists a linear or conjugate-linear isometry $U \colon \HH \to \KK$ such that
				\begin{eqnarray*}
						\phi (A) = U A U^\ast, & A \in \fs.
					\end{eqnarray*}
			\item $k \ge 2$, $\dim \HH = k+m$, and there exists a linear or conjugate-linear isometry $U \colon \HH \to \KK$ such that
		\begin{eqnarray*}
		\phi (A) = U \left(\tfrac{\tr A}{k} \mathcal I_\HH - A\right) U^\ast, & A \in \fs.
		\end{eqnarray*}
				\end{itemize}
		\end{itemize}
\end{theorem}

\begin{proof}
	One implication is clear, so assume that $\phi$ is linear and fulfills \eqref{pphk}.
	Suppose first that $\dim \HH < 2k$. Denote $n = \dim \HH$ and let
	\begin{eqnarray*}
		\psi(A) = \phi \left(\tfrac{\tr A}{n-k} \mathcal I_\HH - A\right), & A \in \fs.
	\end{eqnarray*}
	Then $\dim \HH > 2(n-k)$ and $\psi \colon \fs \to \fs[\KK]$ is a linear map such that $\psi(\PH{n-k}) \subset \PH[\KK]{m}$ and
	\begin{eqnarray*}
		\phi(A) = \psi \left(\tfrac{\tr A}{k} \mathcal I_\HH - A\right), & A \in \fs.
	\end{eqnarray*}
	It is now easy to deduce that if $\psi$ is of any of the above described three forms, then so is $\phi$. We have thus established that it is enough to consider the case $\dim \HH \ge 2k$.
	
	So, suppose that $\dim \HH \ge 2k$. Let $W \in \G{2k}$ and identify $\left\langle W\right]$ with $\HH_{2k}$. By Lemma \ref{phraz} there exist:
	\begin{itemize}
		\item $m', p, r \in \NN \cup \{0\}$, $\tfrac{1}{2} < t_1 < \ldots < t_r < 1$, $m_1, \ldots, m_r \in \NN$,
		\item a unital linear map $\phi_0 \colon \HH_{2k} \to \HH_{2p}$ such that $\phi_0 \left(\ihz[2k]\right) \subset \ihz[2p]$,
		\item real linear maps $\phi_j \colon \hz[2k] \to \mm[m_j]$ such that $\phi_j(\ihz[2k]) \subset \un[m_j]$, $j = 1, \ldots, r$,
		\item $Z \in \G[\KK]{f}$ and its orthonormal basis $\mathcal B$,
	\end{itemize}
	such that for all $A \in \HH_{2k}$ we have $\phi(A) \in \left\langle Z\right]$ and
	\begin{eqnarray*}
		\left(\phi (A)|_Z\right)_{\mathcal B} = \left(\tfrac{\tr A}{k} I_{m'}\right) \oplus \phi_0(A) \oplus \\ \bigoplus_{j=1}^r \begin{bmatrix}
		t_j \tfrac{\tr A}{k} I_{m_j}  & \sqrt{t_j(1-t_j)} \, \phi_j \left(2A - \tfrac{\tr A}{k} I_{2k}\right)^\ast  \\
		\sqrt{t_j(1-t_j)} \, \phi_j \left(2A - \tfrac{\tr A}{k} I_{2k}\right)  & (1-t_j) \tfrac{\tr A}{k} I_{m_j}
	\end{bmatrix}.
	\end{eqnarray*}
It follows from \eqref{pphk} that
\begin{equation}\label{sumk}
	m' + p + \sum_{j=1}^{r} m_j = m \le k.
\end{equation}
If $r > 0$, then Proposition \ref{hura} yields $m_j \ge 2k$, $j = 1, \ldots, r$, contradicting \eqref{sumk}. Thus, $r = 0$.

If $p = 0$, then \eqref{sumk} yields $m' = m$, so there exists $P_0 \in \PH[\KK]{m}$ such that
\begin{eqnarray*}
	\phi(A) = \tfrac{\tr A}{k} P_0, & A \in \left\langle W\right].
\end{eqnarray*}
In this case we say that $W$ is of type (I).

If $p > 0$, then Proposition \ref{hura} implies $p \ge k$, so we obtain from \eqref{sumk} that $m'=0$ and $p=m=k$. Thus, $\dim Z = 2k$. We apply Theorem \ref{iho}, Example \ref{lk2k}, and $\phi_0 \left(I_{2k}\right) = I_{2k}$ to show that there exists a linear or conjugate-linear isometry $U \colon W \to \KK$ such that $\ran U = Z$ and we have either
\begin{eqnarray*}
	\phi(A) = UA|_WU^\ast, & A \in \left\langle W\right],
\end{eqnarray*}
(in which case we say that $W$ is of type (II)) or $k \ge 2$ and
\begin{eqnarray*}
	\phi(A) = \tfrac{\tr A}{k} P_Z - UA|_WU^\ast, & A \in \left\langle W\right],
\end{eqnarray*}
(in which case we say that $W$ is of type (III)).

If $\dim \HH = 2k$, then the proof is complete. Assume now that $\dim \HH > 2k$. For $P, Q \in \PH{k}$ we write $P \perp Q$ if $PQ = 0$ (which is equivalent to $\ran P \perp \ran Q$). We have in particular shown that:
\begin{itemize}
	\item If $W \in \G{2k}$, then $\phi|_{\left\langle W\right]}$ is either injective or constant on $\PH{k} \cap \left\langle W\right]$.
	\item If $P, Q \in \PH{k}$ are such that $P \perp Q$, then we have either $\phi(P) = \phi(Q)$ or $\phi(P) \perp \phi(Q)$.
	\item If $P, Q \in \PH{k}$ are such that $P \not \perp Q$, then $\phi(P) \not \perp \phi(Q)$.
\end{itemize}
We now distinguish three cases.

\setcounter{casec}{0}
\begin{case}{There exists $W_0 \in \G{2k}$ of type (I)}\end{case}
	Let $P_0 \in \PH[\KK]{m}$ such that
	\begin{eqnarray*}
		\phi(A) = \tfrac{\tr A}{k} P_0, & A \in \left\langle W_0\right].
	\end{eqnarray*}
	It remains to show that $\phi(P) = P_0$, $P \in \PH{k}$. So, let $P \in \PH{k} \setminus \left\langle W_0\right]$. Then there exist $Q, R \in \PH{k} \cap \left\langle W_0\right]$ such that $P \perp R$ and $Q \perp R$. Let $P' \in \PH{k}$ such that $\ran P' \subset \ran P + \ran Q$, $P' \ne Q$, and $P' \not \perp Q$. The latter yields $\phi(P') \not \perp \phi(Q) = \phi(R)$. Since $P' \perp R$, this implies $\phi(P') = \phi(R) = P_0$. If $\ran P \not \perp W_0$, then we can choose $Q$ and $R$ above such that $P \not \perp Q$, so we can set $P'$ to be $P$. If $\ran P \perp W_0$, then $P \perp Q$ and $P+Q-P'$ fulfills the above assumptions on $P'$. Consequently, $\phi(P) = \phi(P+Q-P') = P_0$.

\begin{case}{All $W \in \G{2k}$ are of type (II) or (III) and there exists $W_0 \in \G{2k}$ of type (III)}\end{case}
In particular, $k \ge 2$ and
\begin{equation}\label{inj}
	\phi|_{\PH{k}} \textrm{ is injective.}
\end{equation}
There exist $Z_0 \in \G[\KK]{2k}$ and a linear or conjugate-linear isometry $U \colon W_0 \to \KK$ such that $\ran U = Z_0$ and
\begin{eqnarray*}
	\phi(A) = \tfrac{\tr A}{k} P_{Z_0} - UA|_{W_0}U^\ast, & A \in \left\langle W_0\right].
\end{eqnarray*}
Let $W \in \G{2k}$ be adjacent to $W_0$, i.e. $\dim (W_0 \cap W) = 2k-1$.
Then
\begin{eqnarray*}
	\ran P \subsetneq \ran Q \implies \ran \phi(Q) \subsetneq \ran \phi(P), & P, Q \in \PH{} \cap \left\langle W_0 \cap W\right].
\end{eqnarray*}
Since $k > 1$, we have $\dim (W_0 \cap W) > 1$, so $W$ must be of type (III), yielding the existence of $Z \in \G[\KK]{2k}$ and a linear or conjugate-linear isometry $V \colon W \to \KK$ such that $\ran V = Z$ and
\begin{eqnarray*}
	\phi(A) = \tfrac{\tr A}{k} P_Z - VA|_WV^\ast, & A \in \left\langle W\right].
\end{eqnarray*}
Then for all $X \in \G[W_0 \cap W]{k}$ we have
$$
Z \cap V(X)^\bot = \ran \phi \left(P_X\right) = Z_0 \cap U(X)^\bot.
$$
Let $X, X' \in \G[W_0 \cap W]{k}$ such that $\dim(X \cap X') = 1$. Since $k < 2k-1$, there exists $X'' \in \G[W_0 \cap W]{k}$ such that $X \cap X' \cap X'' = \{0\}$. Then
$$
Z = \ran \phi \left(P_X\right) + \ran \phi \left(P_{X'}\right) + \ran \phi \left(P_{X''}\right) = Z_0.
$$
Therefore,
$$
	\phi \left(\PH{k} \cap \left\langle W_0\right]\right) = \PH[\KK]{k} \cap \left\langle Z_0\right] = \PH[\KK]{k} \cap \left\langle Z\right] = \phi \left(\PH{k} \cap \left\langle W\right]\right),
$$
contradicting \eqref{inj}.

\begin{case}{All $W \in \G{2k}$ are of type (II)}\end{case}
Then $\phi \left(\PH{1}\right) \subset \PH[\KK]{1}$. Let $g \colon \G{1} \to \G[\KK]{1}$ be the map such that
\begin{eqnarray*}
	\phi \left(P_X\right) = P_{g(X)}, & X \in \G{1}.
\end{eqnarray*}
Let $X, X', X'' \in \G{1}$. Since there exists $W \in \G{2k}$ such that $X' + X'' \subset W$ and $W$ is of type (II), we have implications
$$
X \subset X' + X'' \implies g(X) \subset g(X') + g(X'')
$$
and
$$
X' \perp X'' \implies g(X') \perp g(X'').
$$
Because $\dim \HH \ge 2k+1 \ge 3$, there exists a linear or conjugate-linear isometry $V \colon \HH \to \KK$ such that
\begin{eqnarray*}
	g(X) = V(X), & X \in \G{1},
\end{eqnarray*}
see e.g. \cite[Theorem 4.1]{faure}, yielding that $\phi(A) = VAV^\ast$, $A \in \fs$.
\end{proof}

\noindent {\bf Acknowledgements.} The author was supported by the Slovenian Research and Innovation Agency program P1-0288.

\bibliographystyle{amsalpha}
\bibliography{reference}
\clearpage
\addcontentsline{toc}{chapter}{Bibliography}

\end{document}